\documentclass[12pt]{article}

\usepackage{latexsym,amsfonts,amssymb,epsfig,verbatim,}
\usepackage{amsmath,amsthm,amssymb,latexsym,graphics,textcomp}
\usepackage{graphicx}

\usepackage{url}

\newtheorem{thm}{Theorem}[section]

\newtheorem{lem}[thm]{Lemma}
\newtheorem{prop}[thm]{Proposition}
\newtheorem{definition}[thm]{Definition}

\newcommand{\BN}{\mathbb N}

\newcommand{\MX}{\mathcal X}
\newcommand{\MP}{\mathcal P}
\newcommand{\MC}{\mathcal C}

\DeclareMathOperator{\Mod}{Mod}

\DeclareMathOperator{\Homeo}{Homeo}
\DeclareMathOperator{\Aut}{Aut}

\newcommand{\mathsym}[1]{{}}
\newcommand{\unicode}[1]{{}}

\begin{document}

\title{Exhausting pants graphs of punctured spheres by finite rigid sets} 

\author{Rasimate Maungchang}

\maketitle

\begin{abstract}
Let $S_{0,n}$ be an $n$-punctured sphere. For $n \geq 4$, we construct a sequence $(\MX_i)_{i\in \BN}$ of finite rigid sets in the pants graph $\MP(S_{0,n})$ such that $\mathcal{X}_1\subset \mathcal{X}_2\subset ...\subset\mathcal{P}(S_{0,n})$ and $\bigcup_{i\geq1}\mathcal{X}_i=\mathcal{P}(S_{0,n})$.
\end{abstract}
\section{Introduction}
\label{sec:Introduction}

Let $S=S_{g,n}$ be an orientable surface of genus $g$ with $n$ punctures and let $\Mod^{\pm}(S)=\pi_0(\Homeo(S))$ be the extended mapping class group. 
Ivanov~\cite{Ivanov}, Korkmaz~\cite{Korkmaz}, and Luo~\cite{Luo} proved that, for most surfaces, the curve complexes $\MC(S)$ is rigid, that is, $\Aut(\MC(S))\cong \Mod^{\pm}(S)$. 
In~\cite{AL}, Aramayona and Leininger proved that curve complexes contain finite rigid sets; meaning a finite subgraph such that every simplicial embedding is a restriction of an element of $\Mod^{\pm}(S)$. Later in~\cite{AL2}, they showed that there exists an exhaustion of the curve complexes by finite rigid sets.

For the pants graphs $\MP(S)$, the rigidity property was proved by Margalit~\cite{Mar} using the result of Ivanov, Korkmaz, and Lou. 
Aramayona~\cite{Aramayona} extended Margalit's result to prove a stronger form of rigidity, that is, if $S$ and $S'$ are surfaces such that the complexity of $S$ is at least $2$, then every injective simplicial map $\phi:\MP(S)\to \MP(S')$ is induced by a $\pi_1$-injective embedding $f:S\to S'$. In~\cite{Rasimate}, we refined Aramayona's result by showing that the pants graphs of punctured spheres are finitely rigid. 

In this paper, we modify the tools Aramayona and Leininger built in~\cite{AL2}, together with the finite rigid sets we constructed~\cite{Rasimate}, to prove that we can exhaust the pants graphs of punctured spheres by finite rigid sets: 
\begin{thm}
	\label{T:main theorem}
	Let $S_{0,n}$ be an $n$-punctured sphere. For $n \geq 4$, there exists a sequence of finite rigid sets $\mathcal{X}_1\subset \mathcal{X}_2\subset ...\subset\mathcal{P}(S_{0,n})$ such that $\bigcup_{i\geq1}\mathcal{X}_i=\mathcal{P}(S_{0,n})$.
\end{thm}

{\bf Outline of the paper.} Section~\ref{sec:background} contains the relevant background and definitions. 
In Section~\ref{sec:enlarge} we describe the adjustments to the tools Aramayona and Leininger~\cite{AL2} developed to enlarge their rigid sets in the curve complex so we can use them in our setting. 
We use these tools to prove the main theorem in Section~\ref{sec:proof}.

{\bf Acknowledgments.}
The author would like to thank Christopher J. Leininger for his guidance, useful conversations, and suggestions.
\section{Background and definitions}
\label{sec:background}
This section contains necessary definitions and background restricted to punctured spheres, for general settings see~\cite{Aramayona} and~\cite{Mar}.
Let $S=S_{0,n}$ be an $n$-punctured sphere. 
A simple closed curve on $S$ is \textbf{essential} if it does not bound a disk or a once-punctured disk on $S$. 
Throughout this paper, a \textbf{curve} is  a homotopy class of essential simple closed curves on $S$.
Given two curves $\gamma$ and $\gamma'$, we denote their \textbf{geometric intersection number} by  $i(\gamma,\gamma')$, which is the minimum number of transverse intersection points among the representatives of $\gamma$ and $\gamma'$. 
The two curves are \textbf{disjoint} if $i(\gamma,\gamma')=0$

A \textbf{multicurve} $Q$ is a set of pairwise distinct, disjoint curves on $S$.
For a given multicurve $Q$, the \textbf{nontrivial piece} $(S-Q)_0$ of the complement of the curves in $Q$ is the union of the non thrice-punctured sphere components of the complement. We call a thrice-punctured sphere, \textbf{a pair of pants}.

A \textbf{pants decomposition} $P$ is a maximal multicurve: the complement in $S$ is a disjoint union of pairs of pants. 
A pants decomposition always contains $n-3$ curves and we call this number the \textbf{complexity} $\kappa(S)$ of $S$. 
The \textbf{deficiency} of a multicurve $Q$ is the number $\kappa(S)-|Q|$. 
If $Q$ is a deficiency-$1$ multicurve then $(S-Q)_0$ is homeomorphic to $S_{0,4}$.

Let $P$ and $P'$ be pants decompositions of $S$. We say that $P$ and $P'$ differ by an  \textbf{elementary move} if there are curves $\alpha, \alpha'$ on $S$ and a deficiency-$1$ multicurve $Q$ such that $P=\{\alpha\}\cup Q, P'=\{\alpha'\}\cup Q$ and $i(\alpha,\alpha')=2$; see Figure~\ref{F:elementary} for an example of elementary moves.

\begin{figure}[ht]
	\begin{center}
		\includegraphics[height=2.5cm]{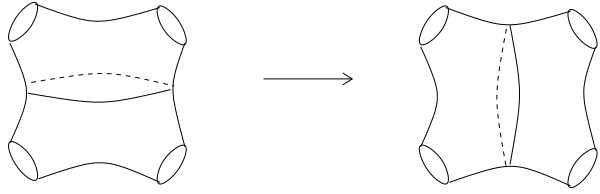} 
		\caption{Example of an elementary move.} 
		\label{F:elementary}
	\end{center}
\end{figure}

The \textbf{pants graph} $\MP(S)$ of $S$ is a graph with the set of vertices corresponding to pants decompositions. 
Two vertices are connected by an edge if the corresponding pants decompositions differ by an elementary move. 
The pants graph $\MP(S)$ is connected and the pants graph $\MP(S_{0,4})$ of a $4$-punctured sphere is isomorphic to a Farey graph, see~\cite{HT}.

A \textbf{path} in $\MP(S)$ is an edge path determined by a sequence of distinct adjacent vertices of $\MP(S)$. 
A \textbf{cycle} in $\MP(S)$ is a subgraph homeomorphic to a circle. 
We call a cycle, a \textbf{triangle}, \textbf{rectangle}, or \textbf{pentagon} if it has $3$, $4$, or $5$ vertices, respectively.

Let $\MX\subset\MP(S_{0,n})$ and $\phi:\MX \to \mathcal{P}(S_{0,m})$ be an injective simplicial map.
We say that a $\pi_1$-injective embedding $f:S_{0,n} \to S_{0,m}$ \textbf{induces} $\phi$ if there is a deficiency-$(n-3)$ multicurve $Q$ on $S_{0,m}$ such that $f(S_{0,n})=(S_{0,m}-Q)_0$ and the simplicial map  
\[
f^{Q}:\mathcal{P}(S_{0,n}) \to \mathcal{P}(S_{0,m}), 
\]
defined by
$
f^{Q}(u)=f(u)\cup Q
$
satisfies $f^{Q}(u)=\phi(u)$ for any vertex $u\in \MX$.
\begin{definition}
	\label{def:rigid}
	For $n\geq4$, we say that $\MX\subset\MP(S_{0,n})$ is \textbf{rigid} if for any punctured sphere $S_{0,m}$ and any injective simplicial map 
	\[
	\phi:\MX \to \mathcal{P}(S_{0,m}),
	\]
	there exists a $\pi_1$-injective embedding $f:S_{0,n} \to S_{0,m}$ that induces $\phi$.
	
	For $n=4$, the isotopy class of $f$ is unique up to precomposing by an element $\sigma\in\Mod(S_{0,4})$ inducing the identity on $\mathcal{P}(S_{0,4})$.
	
	For $n\geq 5$, the isotopy class of $f$ is unique.
\end{definition}
The following theorem is a refinement of Aramayona's result~\cite{Aramayona} that we proved in~\cite{Rasimate}.
\begin{thm}
	\label{T:finite rigid}
	For $n\geq 4$, there exists a finite rigid subgraph $X_n \subset \mathcal{P}(S_{0,n})$.
\end{thm}

\section{Tools for enlarging rigid sets}
\label{sec:enlarge}
This section contains the definitions and theorems Aramayona and Leininger~\cite{AL2} developed to enlarge their rigid sets in the curve complexes. We make some necessary adjustments to them in order to enlarge rigid sets in the pants graphs.
\begin{definition}
	\label{def:weakly}
	Let $n\geq 5$. A set $\MX \subset \MP(S_{0,n})$ is said to be \textbf{weakly rigid} if whenever $f_1, f_2 :S_{0,n}\to S_{0,m}$ are $\pi_1$-injective embeddings satisfy
	\[
	f_1^{Q_1}|_{\MX}=f_2^{Q_2}|_{\MX},
	\]
	for some deficiency-$(n-3)$ multicurves $Q_1$ and $Q_2$ on $S_{0,m}$, then 
	\[Q_1=Q_2 \text{~and~} f_1=f_2,\]
	up to isotopy.
\end{definition}

It is easy to see from the definition that a superset of a weakly rigid set is also weakly rigid.

\begin{lem}
	\label{lem:union of weakly rigid sets}
	For $n\geq 5$, let $\MX_1,\MX_2\subset \MP(S_{0,n})$ be rigid sets. If $\MX_1\cap\MX_2$ is weakly rigid then $\MX_1\cup\MX_2$ is rigid.
\end{lem}
\begin{proof}
	Let $\phi:\MX_1\cup\MX_2\to\MP(S_{0,m})$ be an injective simplicial map. 
	Since $\MX_i$ is rigid, there exist a $\pi_1$-injective embedding $f_i:S_{0,n}\to S_{0,m}$ and a deficiency-$(n-3)$ multicurve $Q_i$ such that $f_i^{Q_i}|_{\MX_i}=\phi|_{\MX_i}$.  
	Hence $f^{Q_1}_1|_{\MX_1\cap\MX_2}=\phi|_{\MX_1\cap\MX_2}=f^{Q_2}_2|_{\MX_1\cap\MX_2}$. 
	The weakly rigidity of $\MX_1\cap\MX_2$ implies that  $Q_1=Q_2=Q$ and $f_1=f_2=f$. 
	Therefore $f$ is a $\pi_1$-injective embedding such that $f^Q|_{\MX_1\cup\MX_2}=\phi$ which implies the rigidity of $\MX_1\cup\MX_2$.
\end{proof}

The following proposition is the key to enlarge rigid sets.

\begin{prop}
	\label{prop:enlarge}
	For $n\geq 5$, let $\MX\subset\MP(S_{0,n})$ be a finite rigid set such that $\Mod(S_{0,n})\cdot\MX=\MP(S_{0,n})$. Suppose there exists a finite subset $C$ of curves on $S_{0,n}$ such that:
	
	$(1)$ The set $\{T^{\pm\frac{1}{2}}_{\alpha}~|~\alpha\in C\}$ generates $\Mod(S_{0,n})$;
	
	$(2)$ $\MX\cap T^i_{\alpha}(\MX)$ is weakly rigid, for all $\alpha\in C$, and $i\in\{-\frac{1}{2},\frac{1}{2}\}$.
	
	Then there exists a sequence $\MX=\MX_1\subset\MX_2\subset...\subset\MX_n\subset...$ such that each $\MX_i$ is a finite rigid set, and
	
	\[
	\bigcup_{i\in\BN}\MX_i=\MP(S_{0,n}).
	\]
\end{prop}
\begin{proof}
	Since $\MX$ is rigid and a half twist is a homeomorphism, $T^{i}_\alpha(\MX)$ is rigid for all $\alpha\in C$, and $i\in\{-\frac{1}{2},\frac{1}{2}\}$.
	Given $\alpha,\beta\in C$ and $i, j\in\{-\frac{1}{2},\frac{1}{2}\}$. 
	By assumption $(2)$ and by applying  Lemma~\ref{lem:union of weakly rigid sets}, we see that $\MX\cup T^i_\alpha(\MX)$ is rigid. 
	Recall that a superset of a weakly rigid set is also weakly rigid. 
	Hence $(\MX\cup T^i_\alpha(\MX))\cap T^j_\beta(\MX)$, which contains $\MX\cap T^j_\beta(\MX)$, is weakly rigid.
	Apply Lemma~\ref{lem:union of weakly rigid sets}, we see that $\MX\cup T^i_\alpha(\MX)\cup T^j_\beta(\MX)$ is weakly rigid. By repeating above arguments, the set $\MX_2:=\MX\cup\bigcup_{\alpha\in C} T^{\pm\frac{1}{2}}_{\alpha}(\MX)$ is rigid. 
	We define
	\[\MX_{n+1}:=\MX_n\cup\bigcup_{\alpha\in C}T^{\pm\frac{1}{2}}_{\alpha}(\MX_n),\]
	for $n\geq2$.
	Since a weakly rigid set $\MX\cap T^i_\alpha(\MX)$ is a subset of $\MX_n\cap T^i_\alpha(\MX_n)$ , $\MX_n\cap T^i_\alpha(\MX_n)$ is weakly rigid.
	Again, by applying~\ref{lem:union of weakly rigid sets} repeatedly and use induction, we conclude that $\MX_n$ is rigid for all $n$. Then the first claim is proved. 
	
	Since  $\{T^{\pm\frac{1}{2}}_{\alpha}~|~\alpha\in C\}$ generates $\Mod(S_{0,n})$ and $\Mod(S_{0,n})\cdot\MX=\MP(S_{0,n})$,
	\[
	\bigcup_{i\in\BN}\MX_i=\MP(S_{0,n}).
	\]	
\end{proof}

\section{The proof of the main theorem}
\label{sec:proof}
We note that for $n\leq 3$, the pants graphs $\MP(S_{0,3})$ is empty. 
We give a separate proof for $n=4$ as follow.

\begin{proof}[Proof of Theorem~\ref{T:main theorem} for $S=S_{0,4}$.]
	The pants graph of $S_{0,4}$ is isomorphic to the Farey graph. 
	Any triangle in $S_{0,4}$ is rigid as proved in~\cite{Rasimate}. 
	Then we let $\MX_1$ to be a triangle. 
	Each edge in a pants graph of any punctured sphere is contained in exactly two triangles which are both in the same Farey graph.
	Then we can define $\MX_{n+1}$ inductively; let $\MX_{n+1}$ be an enlargement of $\MX_n$ obtained by attaching one more triangle to each edge of $\MX_n$ contained in only one triangle.
	Hence $\MX_{n+1}$ is rigid for all $n\geq 1$, and by the construction, $\bigcup_{i\in\BN}(\MX_i)=\MP(S_{0,4})$. 
	We conclude that sequence $(\MX_n)_{n\in\BN}$ is an exhaustion of $\MP(S_{0,4})$.
\end{proof}  

For $n\geq 5$, we begin by recalling the construction of finite rigid sets $X_n$ in~\cite{Rasimate}. 
First we construct $S_{0,n}$ with a set of curves, then define $X_5$, and finally, define $X_n$ for $n\geq 6$.

Consider a regular $n$-gon with the $n$ vertices removed and label the sides as $1, 2,..., n$, cyclically.
For each non-adjacent pair of sides of the $n$-gon, draw a straight line segment to connect the two sides.  
Then double the $n$-gon to obtain $S_{0,n}$ and a set of curves $\Gamma_n$, see Figure~\ref{F:S_0_8} for the case of $S_{0,8}$ and Figure~\ref{F:S05andX05} for the case of $S_{0,5}$.
Let $a_{i,j}\in\Gamma_n$ be the curve connecting the $i^{th}$ side to the $j^{th}$ side of $S_n$. We call $a_{i,j}$ such that $i-j\equiv\pm 2\mod n$, \textbf{a chain curve}. Compare to \cite[Section 3]{AL}.

Let $Z_n$ be a subgraph of $\mathcal{P}(S_{0,n})$ induced by the set of vertices corresponding to  pants decompositions consisting of curves from $\Gamma_n$.

For $\MP(S_{0,5})$, we defined
\[X_5 = Z_5 \cup \bigcup_{c \in \Gamma_5} T^{\pm \frac{1}{2}}_c(Z_5),\]
where $T^{\frac{1}{2}}_c$ is a simplicial map on  $\MP(S_{0,5})$ induced by the half-twist around the curve $c$. 

See Figure~\ref{F:S05andX05} for a partial figure of $X_5$.
The subgraph $X_5$ consists of the alternating pentagon $Z_5$ and $10$ of its images under the twists. 
Those $10$ images form $10$ triangles attached to $Z_5$. 
In~\cite{Rasimate}, we proved that $X_5$ is rigid.

For $n\geq 6$, we construct $X_n$ as follows.
Let $W\subset \Gamma_n$ be a deficiency-$2$ multicurve such that $(S_{0,n}-W)_0\cong S_{0,5}$. 
Let $\Gamma^W_5=\{\alpha\in\Gamma_n~|~\alpha$ is disjoint from all curves in $W\}$.
There is a natural homeomorphism $h:S_{0,5} \to (S_{0,n}-W)_0$ such that $h(\Gamma_5)=\Gamma_5^W$, see~\cite[Lemma 3.1]{Rasimate}.
Let 
\[ X_5^W = h^W(X_5) = \{ h(u) \cup W \mid u \in X_5 \},\]
where $h^W:P(S_{0,5})\to P(S_{0,n})$ is the induced map of $h$ defined by $h^W(u)=h(u)\cup W$. 
Then $X_5^W\cong X_5$.
Finally we let 
\[
X_n = Z_n \cup \bigcup_W X^W_5,
\]
where the union is taken over all deficiency-$2$ multicurves in $\Gamma_n$ with a $5$-punctured sphere component.
In~\cite{Rasimate}, we proved that $X_n$ is rigid.
\begin{figure}[ht]
	\begin{center}
		\includegraphics[height=4cm]{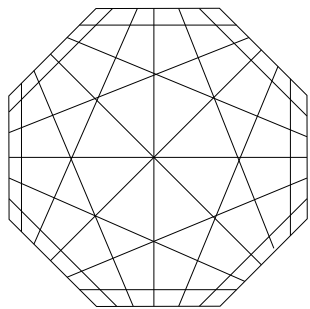} 
		\caption{$S_{0,8}$ and the set of simple closed curves $\Gamma_8$.} 
		\label{F:S_0_8}
	\end{center}
\end{figure}
\begin{figure}[ht]
	\begin{center}
		\includegraphics[height=9cm]{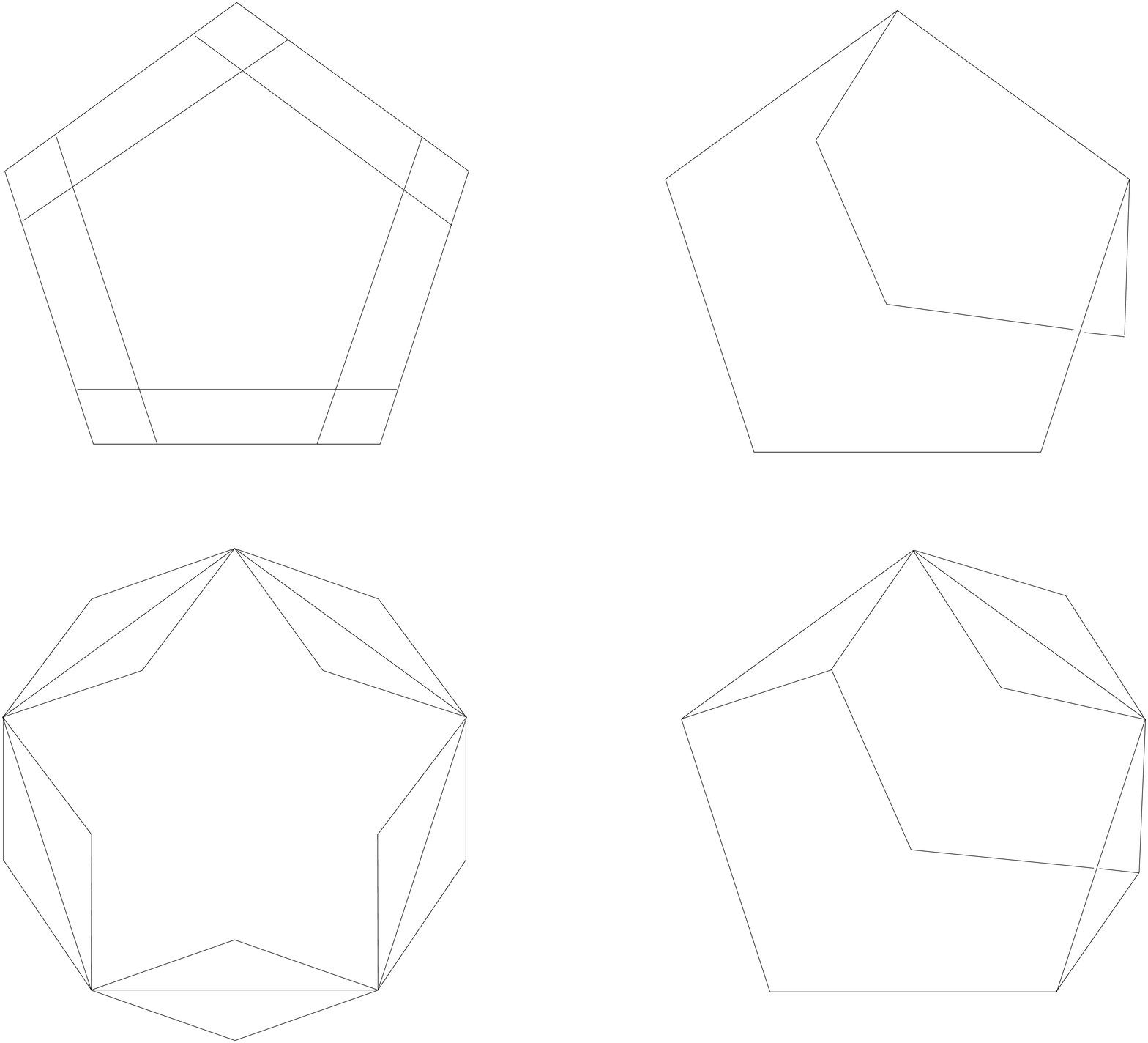} 
		\caption{(Top left) $S_{0,5}$ and curves in $\Gamma_5$, (top right) $Z_5 \cup T^{\frac{1}{2}}_\beta(Z_5)$, (bottom left) $Z_5$ together with the $10$ triangles, and (bottom right) $X_5\cap T^{\frac{1}{2}}_\alpha(X_5)$.} 
		\label{F:S05andX05}
	\end{center}
	
	\begin{picture}(0,0)(0,0)
	\put(95,325){\small $\alpha$}
	\put(60,300){\small $\beta$}
	\put(50,273){\small $\epsilon$}
	\put(62,230){\small $\gamma$}
	\put(127,207){\small $\delta$}
	
	\put(270,330){\scriptsize $A=\left\{\alpha,\beta\right\}$}
	\put(182,290){\scriptsize $E=\left\{\alpha,\gamma\right\}$}
	\put(330,290){\scriptsize $B=\left\{\delta,\beta\right\}$}
	\put(310,208){\scriptsize $C=\left\{\delta,\epsilon\right\}$}
	\put(235,208){\scriptsize $D=\left\{\gamma,\epsilon\right\}$}
	
	\put(330,235){\scriptsize $T^{1/2}_{\beta}(C)$}
	\put(250,245){\scriptsize $T^{1/2}_{\beta}(D)$}
	\put(260,290){\scriptsize $T^{1/2}_{\beta}(E)$}
	\end{picture}
\end{figure}

	We need the following lemmas to prove the main theorem for $n\geq 5$.
	\begin{lem}
		\label{lem:Mod}
		$\Mod(S_{0,n})\cdot X_n=\MP(S_{0,n})$
	\end{lem}
	\begin{proof}
		In the first part of this proof, we will show that, for a given vertex $P$ in $\MP(S_{0,n})$, there exist a vertex $P'$ in $X_n$ and $f\in\Mod(S_{0,n})$ such that $f(P')=P$. 
		To do this, we obtain a pants decomposition $P'$ from a dual graph of the pants decomposition $P$. 
		For the second part, we will show that there is a homeomorphism that send a given edge in $\MP(S_{0,n})$ to an edge in $Z_n\subset X_n$. 
		
		Given a vertex $P$ in $\MP(S_{0,n})$. 
		Recall that we consider $S_{0,n}$ as a double of a regular $n$-gon. 
		Consider $P$ as a pants decomposition on $S_{0,n}$. 
		The following construction of a dual graph of $P$ was given in~\cite{HT}. For each pair of pants component of $(S_{0,n}-P)$, we mark a vertex on the interior of the component. 
		We also mark the $n$ punctures as $n$ vertices. Two vertices are connected by an edge if (1) they are vertices on the interior of two pants components which have a common boundary, or (2) one of the vertices is on the interior of a pair of pants component and another vertex is a puncture of the same component. 
		The result is a tree with $2n-2$ vertices; all puncture-vertices have degree $1$, while the rest of the vertices have degree $3$, see Figure~\ref{F:dual}.
		
		\begin{figure}[ht]
			\begin{center}
				\includegraphics[height=4.5cm]{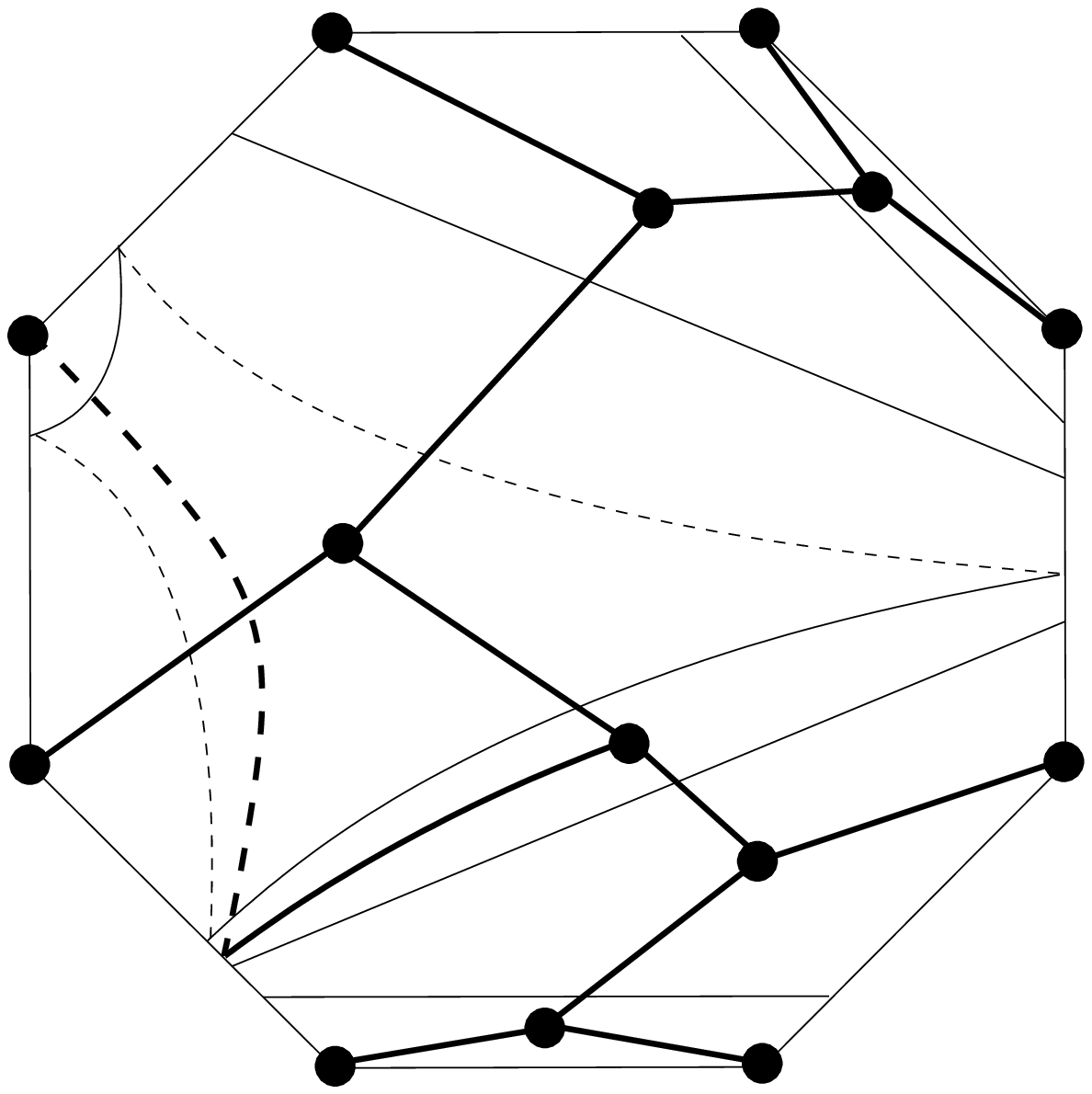} 
				\caption{Example of a pants decomposition of $S_{0,8}$ and its dual graph shown in thick edges.} 
				\label{F:dual}
			\end{center}
		\end{figure}
		
		Since a tree is planar, we can redraw this tree on the plane inside a regular $n$-gon so that all $n$ puncture-vertices are the $n$ vertices of the $n$-gon. 
		We reconstruct a pants decomposition consisting of curves in $\Gamma_n$ by drawing a curve connecting two sides of the regular $n$-gon whenever this curve can cross exactly one edge of the tree and both endpoints of this edge are not puncture-vertices. Double the regular $n$-gon. 
		We now have a pants decomposition $P'$ consisting of curves in $\Gamma_n$, i.e., $P'$ is a vertex in $Z_n\subset X_n$. 
		
		The above construction of $P'$ from $P$ gives a one-to-one correspondence between the pants components $S_{0,n}-P$ and the pants components $S_{0,n}-P'$. 
		This correspondence describes a homeomorphism $f$ such that $f(P')=P$, as desired.
		
		Next we show that if $P_1$ and $P_2$ are adjacent vertices in $\MP(S_{0,n})$, then after applying some homeomorphisms on $S_{0,n}$ to $P_1$ and $P_2$, we get two vertices that are adjacent in $Z_n$.
		
		Given adjacent vertices $P_1$ and $P_2$ in $\MP(S_{0,n})$, then there exist curves $u_1, u_2$ on $S_{0,n}$ and a deficiency-$1$ multicurve $Q$ such that $P_1=\{u_1\}\cup Q$ and $P_2=\{u_2\}\cup Q$. 
		By the first part of the proof, there is $f\in\Mod(S_{0,n})$ such that $f(P_1)$ is a vertex in $Z_n$. 
		If $f(P_2)$ is also in $Z_n$, then we are done.
		
		Suppose $f(P_2)$ is not in $Z_n$. 
		Use Figure~\ref{F:full} as a reference for the rest of the proof. 
		We note that $f(Q)\subset\Gamma_n$ and it has deficiency-$1$. The nontrivial component $(S_{0,n}-f(Q))_0\cong S_{0,4}$ contains exactly two curves in $\Gamma_n$; one curve is $f(u_1)$ and we call the other curve $\alpha$. 
		Then $i(f(u_2),\alpha)=2n$ for some $n\in\BN$.
		Applying one full twist around $f(u_1)$ in an appropriate direction reduces the intersection number by $4$. 
		$f(P_1)$ is invariant under this full twist. 
		So we can choose a new $f$ (by composing the old one with some power of full twists) and assume that $i(f(u_2),\alpha)=0$ or $i(f(u_2),\alpha)=2$. 
		If $i(f(u_2),\alpha)=0$, then $f(u_2)=\alpha$ and we are done.
		
		Suppose $i(f(u_2),\alpha)=2$. 
		We compose $f$ by an appropriate \textit{half twist} $T$ around $f(u_1)$: here a half twist in $f(u_1)$ is a homeomorphism on $S_{0,n}$, whose square is the Dehn twist in $f(u_1)$,  although we note that it does not necessary restrict to a homeomorphism of $(S_{0,n}-f(Q))_0\cong S_{0,4}$. We choose the half twist that essentially ``flips over'' half of the $n$-gon, cut along $f(u_1)$; see  Figure~\ref{F:halftwist} and also  Figure~\ref{F:full}. 
		Then $T\circ f(u_2)=\alpha$ and the edge $\{T\circ f(P_1),T\circ f(P_2)\}$ is in $Z_n$ as desired. 
	\end{proof}
	
	\begin{figure}[ht]
		\begin{center}
			\includegraphics[height=7cm]{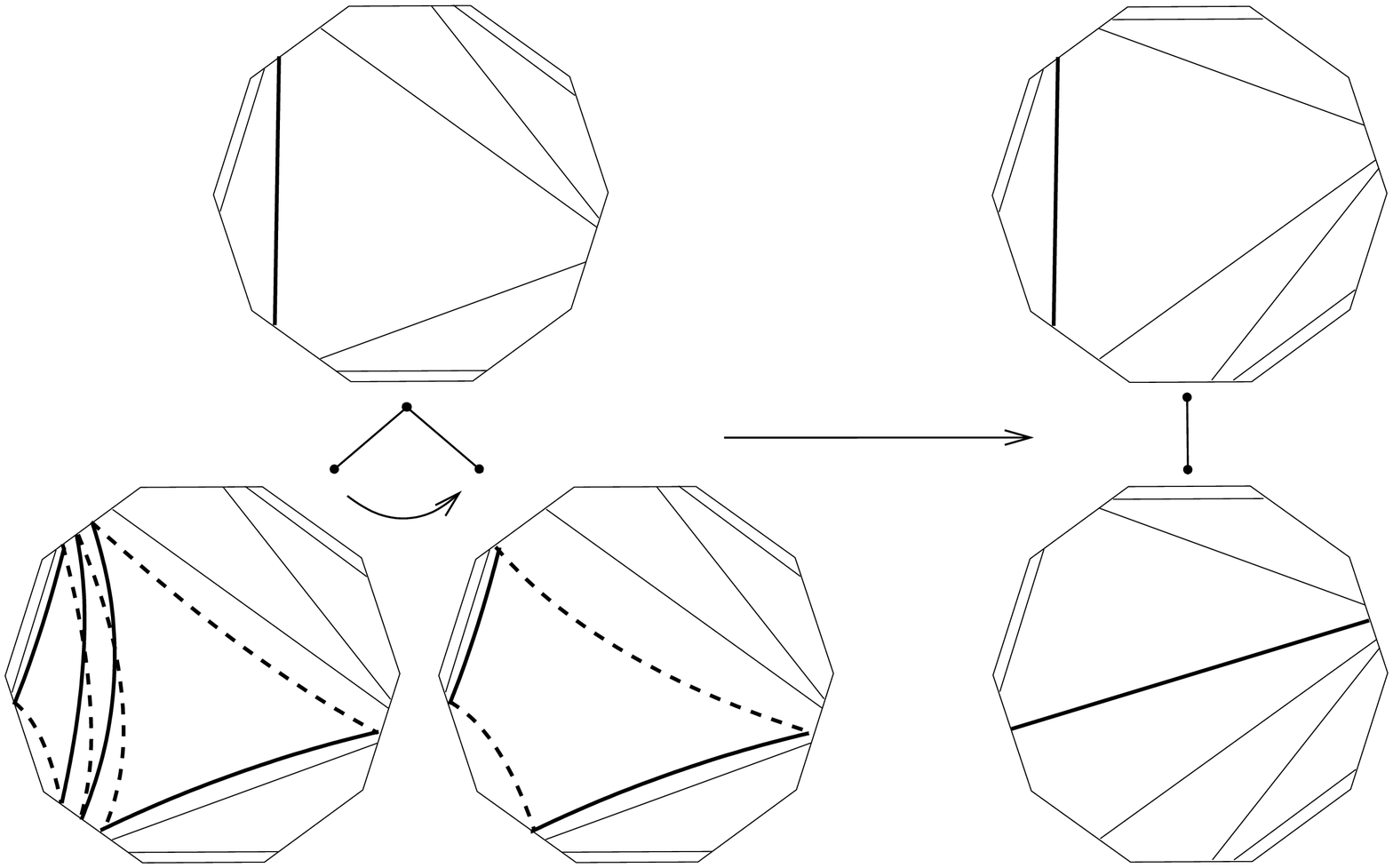} 
			\caption{Example of an edge $\{f(P_1),f(P_2)\}$ and its images after composing with a power of full twist around the curve $f(u_1)$ and a \textit{half twist} around the same curve.} 
			\label{F:full}
		\end{center}
		\begin{picture}(0,0)(0,0)
		\put(80,275){\small $f(u_1)$}
		\put(33,165){\small $f(u_2)$}
		\end{picture}
	\end{figure}
	
	\begin{figure}[ht]
		\begin{center}
			\includegraphics[height=7cm]{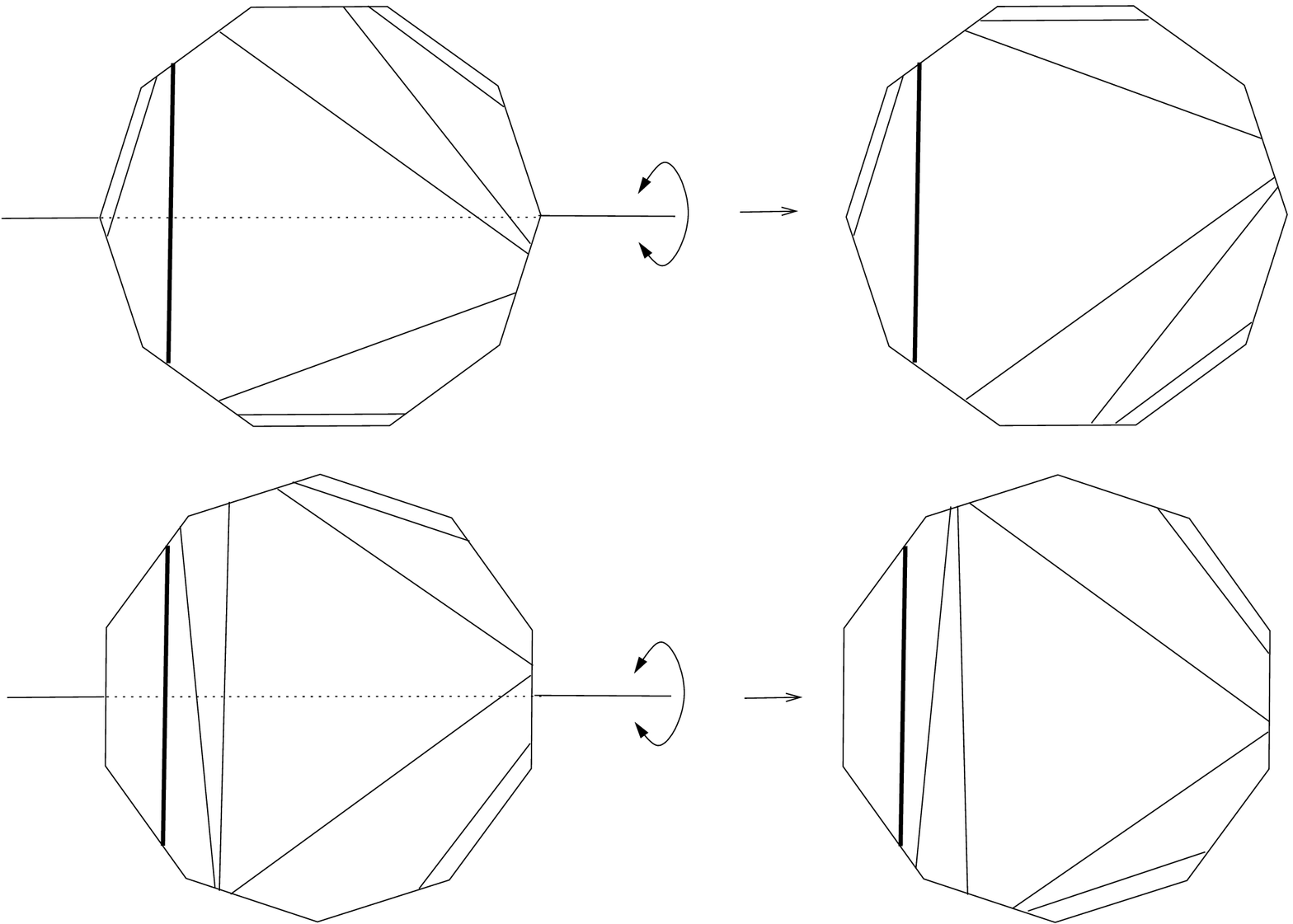} 
			\caption{Examples of half twist around the thick curves. Two pants decompositions in $Z_{10}$ and $Z_{11}$ are given to help visualize the homeomorphisms. Note that after a half twisting, we get a new pants decomposition that is still in $Z_{10}$ or $Z_{11}$.} 
			\label{F:halftwist}
		\end{center}
	\end{figure}
	
	Let $\alpha$ be a curve on $S_{0,n}$. We defined $\MP_\alpha(S_{0,n})$ to be a subgraph of $\MP(S_{0,n})$ induced by vertices corresponding to pants decompositions containing $\alpha$.
	
	The following lemma is proved in~\cite{Rasimate} and we use this lemma to prove Lemma~\ref{lem:weakly2}.
	\begin{lem}
		\label{lem:X_n}
		For $n\geq6$, let $\alpha$ be a chain curve on $S_{0,n}$ and let $X^\alpha_{n-1}=X_n \cap \mathcal{P}_{\alpha}(S_{0,n})$. 
		
		Then $X^\alpha_{n-1}\cong X_{n-1}$. 
		Moreover, this isomorphism is induced by $h: S_{0,n-1}\to(S_{0,n}-\alpha)_0$ as $h^{\alpha}(v) = h(v) \cup \{\alpha\} \in X_{n-1}^\alpha$.
	\end{lem}
	
	\begin{lem}
		\label{lem:weakly2}
		$X_n\cap T^i_\alpha(X_n)$ is weakly rigid, for $i\in\{-\frac{1}{2},\frac{1}{2}\}$ and for all chain curves $\alpha$ in $S_{0,n}$.
	\end{lem}
	\begin{proof}
		Let $\alpha$ be a chain curve and $i\in\{-\frac{1}{2},\frac{1}{2}\}$. Suppose $f_1, f_2 :S_{0,n}\to S_{0,m}$ are $\pi_1$-injective embeddings such that
		\[
		f_1^{Q_1}|_{X_n\cap T^i_\alpha(X_n)}=f_2^{Q_2}|_{X_n\cap T^i_\alpha(X_n)},
		\]
		for some deficiency-$(n-3)$ multicurves $Q_1$ and $Q_2$ on $S_{0,m}$.
		
		We first prove the case of $n=5$. Recall the definition of $X_5$. 
		By a direct calculation, we see that $X_5\cap T^i_\alpha(X_5)$ consists of two alternating pentagons which are $Z_5=T^i_\alpha(T^{-i}_\alpha(Z_5))$ and $T^i_\alpha(Z_5)$. 
		They share an edge together with four triangles as shown in Figure~\ref{F:S05andX05}. 
		Since $Z_5$ is an alternating pentagon and $f^{Q_1}_1|_{Z_5}=f^{Q_2}_2|_{Z_5}$,  \cite[Lemma 4.2]{Luo} implies that $Q_1=Q_2$ and 
		\[f_1=f_2 \text{ or }  f_1=f_2\circ e,\]
		where $e:S_{0,5}\to S_{0,5}$ is the involution exchanging the two pentagons (as we consider $S_5$ as a double of a pentagon). 
		The map $e$ induces a simplicial map on $\MP(S_{0,5})$ that fixes $Z_5$ and exchanges two triangles on each side of $Z_5$. 
		But $f_1$ and $f_2$ also agree on the four triangles attached to $Z_5$ so $f_1=f_2$. 
		Hence the case of $n=5$ is proved.
		
		Let $n\geq 6$ and let $\alpha$ be any chain curve. 
		By Lemma~\ref{lem:X_n}, a subgraph $X^\alpha_{n-1}=X_n\cap\MP_{\alpha}(S_{0,n})\cong X_{n-1}$. 
		Since each vertex of $X^\alpha_{n-1}$ contains $\alpha$, $T^i_{\alpha}(X^\alpha_{n-1})=X^\alpha_{n-1}$. Hence $X_n\cap T^i_\alpha(X_n)$ contains $X^\alpha_{n-1}\cong X_{n-1}$. 
		Consider the restrictions of $f_1$ and $f_2$ on the subsurface $(S_{0,n}-\{\alpha\})_0$. 
		Since $X_{n-1}$ is rigid, so is $X^{\alpha}_{n-1}$, and the uniqueness part of Definition~\ref{def:rigid} implies that $f_1$ agrees with $f_2$ on $(S_{0,n}-\{\alpha\})_0$ and $Q_1\cup\{f_1(\alpha)\}=Q_2\cup\{f_1(\alpha)\}$.
		
		We can see that $X^\alpha_{n-1}$ is a proper subgraph of $X_n\cap T^i_\alpha(X_n)$. 
		For example, choose a vertex $P$ in $Z_n\cap\MP_{\alpha}(S_{0,n})\subset X^\alpha_{n-1}$. 
		Then change $P$ to $P'$ by the elementary move which replaces $\alpha$ by the other curve $\alpha'$ in $\Gamma_n$. 
		The vertex $T^i_{\alpha}(P')$ is adjacent to $P$ and it is a vertex in $X_n\cap T^i_\alpha(X_n)$. 
		Hence $f_1$ and $f_2$ agree on $T^i_{\alpha}(P')$. 
		Since $Q_1$ and $Q_2$ are the intersections of all vertices in $f_1(X_n\cap T^i_\alpha(X_n))$ and $f_2(X_n\cap T^i_\alpha(X_n))$, respectively, and $\alpha\notin T^i_{\alpha}(P')$, it follows $f_1(\alpha)=f_2(\alpha)$ is not in the intersection. Therefore $Q_1=Q_2$ and $f_1=f_2$. 
	\end{proof}

\begin{proof}[Proof of Theorem~\ref{T:main theorem} for $S_{0,n}, n\geq5$]
	We are ready to prove the main theorem for $n\geq 5$. 
	We check that all conditions in Proposition~\ref{prop:enlarge} are satisfied.
	Let $\MX=X_n$. 
	Lemma~\ref{lem:Mod} states that $\Mod(S_{0,n})\cdot\MX=\MP(S_{0,n})$. 
	The set 
	\[
	C=\{T^{\pm\frac{1}{2}}(\alpha)~|~\alpha \text{~a chain curve}\}
	\]
	generates $\Mod(S_{0,n})$, see \cite[Corollary 4.15]{danprime}. 
	And by Lemma~\ref{lem:weakly2}, $X_n\cap T^i_\alpha(X_n)$ is weakly rigid, for $i\in\{-\frac{1}{2},\frac{1}{2}\}$ and for all chain curves $\alpha$ in $S_{0,n}$. 
	Therefore Proposition~\ref{prop:enlarge} gives us a sequence of finite rigid set $\MX=\MX_1\subset\MX_2\subset...\subset\MX_m\subset...$ such that $\bigcup_{i\in\BN}\MX_i=\MP(S_{0,n})$, as desired. 
	
\end{proof}	


\bibliographystyle{plain}
\bibliography{paper3}

\begin{thebibliography}{10}

\bibitem{Aramayona}
J.~Aramayona.
\newblock Simplicial embeddings between pants graphs.
\newblock {\em Geometriae Dedicata}, 144(1):115--128, 2010.

\bibitem{AL}
J.~Aramayona and C.~J. Leininger.
\newblock Finite rigid sets in curve complexes.
\newblock {\em Journal of Topology and Analysis}, 5(2):183--203, 2013.

\bibitem{AL2}
J.~Aramayona and C.~J. Leininger.
\newblock Exhausting curve complexes by finite rigid sets.
\newblock {\em Pacific Journal of Mathematics}, 282(2):257--283, 2016.

\bibitem{danprime}
B.~Farb and D.~Margalit.
\newblock {\em A Primer on Mapping Class Groups (PMS-49)}.
\newblock Princeton University Press, 2011.

\bibitem{HT}
A.~Hatcher and W.~Thurston.
\newblock A presentation for the mapping class group of a closed orientable
  surface.
\newblock {\em Topology}, 19(3):221--237, 1980.

\bibitem{Ivanov}
N.~V. Ivanov.
\newblock Automorphism of complexes of curves and of teichm{\"u}ller spaces.
\newblock {\em International Mathematics Research Notices}, 1997(14):651--666,
  1997.

\bibitem{Korkmaz}
M.~Korkmaz.
\newblock Automorphisms of complexes of curves on punctured spheres and on
  punctured tori.
\newblock {\em Topology Appl.}, 95(2):85--111, 1999.

\bibitem{Luo}
F.~Luo.
\newblock Automorphisms of the complex of curves.
\newblock {\em Topology}, 39(2):283--298, 2000.

\bibitem{Mar}
D.~Margalit.
\newblock Automorphisms of the pants complex.
\newblock {\em Duke Mathematical Journal}, 121(3):457--479, 2004.

\bibitem{Rasimate}
R.~Maungchang.
\newblock Finite rigid subgraphs of the pants graphs of punctured spheres.
\newblock {\em arXiv preprint arXiv:1303.3873}, 2013.

\end{thebibliography}


\end{document}